\documentclass[12pt,leqno]{article}
\usepackage{amsmath, amsthm, amsfonts, amssymb, color}
\usepackage{mathrsfs}
\newtheorem{thm}{Theorem}[section]
\newtheorem{cor}[thm]{Corollary}

\newtheorem{prp}[thm]{Proposition}

\theoremstyle{definition}

\newcommand{\scr}[1]{\mathscr #1}
\definecolor{wco}{rgb}{0.5,0.2,0.3}

\numberwithin{equation}{section} \theoremstyle{remark}

\newcommand{\ua}{\uparrow}

\title{{\bf Harnack Inequalities and Applications for Ornstein-Uhlenbeck Semigroups with Jump}\footnote{Supported in
 part by WIMICS, NNSFC(10721091), the 973-Project, the DFG through
SFB-701 and IRTG 1132.} }
\author{
{\bf Shun-Xiang Ouyang$^{a),b)}$, Michael R\"ockner$^{b)}$, Feng-Yu Wang$^{a),c)}$}\footnote{Corresponding author.
wangfy@bnu.edu.cn; F.Y.Wang@swansea.ac.uk}\\
\footnotesize{$^{a)}$School of Math. Sci. \& Lab. Math. Com. Sys.,
Beijing Normal
University, Beijing 100875, China}\\
\footnotesize{$^{b)}$Department of Mathematics, Bielefeld
University, D-33501 Bielefeld, Germany}\\
\footnotesize{$^{c)}$Department of Mathematics, Swansea University,
Singleton Park, SA2 8PP, UK} }
\begin{document}
\def\R{\mathbb R}  \def\ff{\frac} \def\ss{\sqrt} \def\B{\mathbf
B} \def\BB{\scr B}
\def\N{\mathbb N} \def\kk{\kappa} \def\m{{\bf m}}
\def\dd{\delta} \def\DD{\Delta} \def\vv{\varepsilon} \def\rr{\rho}
\def\<{\langle} \def\>{\rangle} \def\GG{\Gamma} \def\gg{\gamma}
  \def\nn{\nabla} \def\pp{\partial} \def\tt{\tilde}
\def\d{\text{\rm{d}}} \def\bb{\beta} \def\aa{\alpha} \def\D{\scr D}
\def\EE{\mathbb E} \def\si{\sigma} \def\ess{\text{\rm{ess}}}
\def\beg{\begin} \def\beq{\begin{equation}}  \def\F{\scr F}
\def\Ric{\text{\rm{Ric}}} \def\Hess{\text{\rm{Hess}}}
\def\e{\text{\rm{e}}} \def\ua{\underline a} \def\OO{\Omega}  \def\oo{\omega}
 \def\tt{\tilde} \def\Ric{\text{\rm{Ric}}}
\def\cut{\text{\rm{cut}}} \def\P{\mathbb P} \def\ifn{I_n(f^{\bigotimes n})}
\def\C{\scr C}      \def\aaa{\mathbf{r}}     \def\r{r}
\def\gap{\text{\rm{gap}}} \def\prr{\pi_{{\bf m},\varrho}}  \def\r{\mathbf r}
\def\Z{\mathbb Z} \def\vrr{\varrho} \def\ll{\lambda}
\def\L{\scr L}\def\Tt{\tt} \def\TT{\tt}\def\II{\mathbb I}
\def\i{{\rm i}}\def\Sect{{\rm Sect}}\def\E{\mathbb E} \def\H{\mathbb H}
\def\M{\scr M}\def\Q{\mathbb Q} \def\texto{\text{o}} \def\i{{\rm i}}

\maketitle
\begin{abstract}  The Harnack inequality established in \cite{RW03}
for generalized Mehler semigroup is improved and generalized. As
applications, the log-Harnack inequality, the strong Feller
property, the hyper-bounded property, and some heat kernel
inequalities are presented for a class of O-U type semigroups with
jump. These inequalities and semigroup properties are indeed
equivalent, and thus sharp, for the Gaussian case. As an application
of the log-Harnack inequality, the HWI inequality is established for
the Gaussian case. Perturbations with linear growth are also
investigated.
\end{abstract} \noindent

 \vskip 2cm

\section{Introduction}
In this paper we aim to establish Harnack inequalities and
applications for a class of Ornstein-Uhlenbeck type SDEs driven by
L\'evy noises on Hilbert spaces. This problem has been investigated
in \cite{RW03} by using Mehler type formula for the associated
semigroups and gradient estimates for dimension-free
Harnack inequalities  developed in \cite{W97}. In this paper, we
shall adopt a measure transformation argument to derive a more
general and sharper Harnack inequality, and to present finer
estimates  of heat kernels. This method was initiated in
\cite{ATW06} by using coupling and Girsanov transformation to
establish Harnack inequalities for diffusion semigroups on manifolds
with unbounded below curvature, and has been applied in \cite{W07,
LW08, DRW09} for non-linear SPDEs driven by Gaussian
noises and also in \cite{OY} for diffusions with singular drifts and multivalued 
stochastic evolution equations. In this paper we shall
modify this argument to SPDEs with jumps.

Let us first recall the Harnack inequality derived in \cite{RW03}.
Let $\H$ be a real separable Hilbert space with inner product
$\<\cdot,\cdot\>$ and norm $\|\cdot\|$. Consider the following
L\'evy driven stochastic differential equation
\begin{equation}\label{1.1}
        \d X_t=AX_t\d t+\d Z_t,\quad X_0=x\in\H,
\end{equation}
where $A$ is the infinitesimal generator of a strongly continuous
  semigroup \((T_t)_{t\geq 0}\) on \(\H\), 
\(Z_t:=\{Z_t^u,\ u\in\H\} \) is a cylindrical L\'evy process with characteristic triplet \((a, R, M)\)
on some filtered probability space $(\Omega, \F, (\F_t)_{t\geq 0},
\P)$, that is, for every $u\in\H$ and $t\geq 0$
\begin{equation*}
\begin{split}
\E\exp(\i \< Z_t, u\>)=\exp \Bigl[
& \i t\<a, u\>-\frac{t}{2}\<Ru,u\>\\
        &{}-\int_{\H} \left[1-\exp(\i\<x,u\>)+\i\<x,u\>1_{\{|x|\leq 1\}}(x)\right] \,M(\d x)\Bigr],
\end{split}
\end{equation*}
where $a\in\H$, $R$ is a symmetric  linear operator on $\H$ such
that
$$R_t:= \int_0^t T_s RT_s^*\,\d s$$ is trace class for each $t>0,$ 
and $M$ is a L\'evy measure on $\H$.
(For simplicity, we shall write $Z_t^u=\<Z_t,u\>$ for every $u\in\H$.)
In this case, \eqref{1.1} has a unique mild solution
$$X_t= T_t x+\int_0^t T_{t-s}\d Z_s,\ \ t\ge 0.$$
Let
$$ P_tf(x)=\E f(X_t), \quad x\in\H,\ f\in \BB_{b}(\H),
$$
where $\BB_{b}(\H)$ is the space of all bounded measurable functions
on $\H$. Similarly, let  $\BB_b^+(\H)$,   $\C_b(\H)$, $\C^\infty(\H)$ be
 the classes of bounded positive,   bounded continuous, and
smooth functions on $\H$ respectively. Let $G$ be the orthogonal
complement of $\text{Ker} R^{1/2}$. Then the inverse $R^{-1/2}$ of
$R^{1/2}$ is well defined from $R^{1/2} \H$ to $G$. The following is
the main result derived in \cite{RW03}.

\beg{thm} \label{T0}{\bf (\cite{RW03})} Assume that there exists a
sequence of eigenvectors of $A^*$ separating the points of $\H$, $R$ is
of trace class, and $T_t R\H\subset R^{1/2} \H$ holds for all $t>0$.
If
\begin{equation}\label{*1} \|R^{-1/2}T_t Rx\|\le \ss{h(t)}\,\|R^{1/2}x\|,\ \
x\in \H, t\ge 0
\end{equation} 
holds for some positive function $h\in
C[0,\infty).$ Then for any $f\in \BB_b^+(\H),$
\beq\label{*2} (P_{t}f)^{2}(x)
    \leq \exp\bigg[\frac{\|R^{-1/2}(x-y)\|^2}{\int_0^t h(s)^{-1}\d s}
    \bigg]
    P_{t}f^2(y),\ \  t>0, x-y\in R^{1/2}\H.\end{equation}
If $M=0$, then for any $\aa>1$ and $f\in \BB_b^+(\H),$
\beq\label{*3} (P_{t}f)^{\alpha}(x)
    \leq  \exp\bigg[\frac{\aa\|R^{-1/2}(x-y)\|^2}{2(\aa-1)\int_0^t h(s)^{-1}\d s}
    \bigg]  P_{t}f^{\alpha}(y), \ \   t>0, x-y\in
    R^{1/2}\H.
\end{equation} 
\end{thm}

We note that due to the absence of a chain rule, for the case with jump
(i.e. $M\ne 0$), the Harnack inequality was
proved in \cite{RW03} only for $\alpha=2$ (i.e. (1.3)) 
by using gradient estimates.

To improve this result, we shall adopt a measure transformation
argument and the null controllability of the associated
deterministic equation (see Section 2). As a result, we obtain the
Harnack inequality by using the image norm $\|\GG_t x\|$ of the
operator

$$\GG_t:= R_t^{-1/2} T_t\ \text{with\ domain}\
 \D(\GG_t):= \big\{x\in \H:\ T_t x\in R_t^{1/2}\H\big\}.$$
 As explained above,
$R_t^{-1/2}$ is defined from $R_t^{1/2}\H$ to the orthogonal
complement of $\text{Ker}\, R_t^{1/2}.$ By letting $\|\GG_t
x\|=\infty$ for $x\notin \D(\GG_t)$ and $\inf\emptyset =\infty$, we
have

$$\|\GG_tx\|= \inf\big\{\|z\|:\ z\in \H, R_t^{1/2}z= T_t x\big\},\ \
x \in\H.$$

Our first result is an improvement of Theorem \ref{T0}: our Harnack
inequality generalizes (\ref{*2}) without the assumptions of Theorem
\ref{T0}. Moreover, our argument also implies the following
inequality (\ref{Gr}), which in particular implies the strong Feller
property 
(even $\|\Gamma_t\cdot\|$-Lipschitz strong Feller property)
of $P_t$ if $\GG_t$ is bounded.

\beg{thm} \label{T1} For any $\aa>1$ and $f\in\BB_b^+(\H)$,
\beq\label{H1} (P_t f(x))^\aa \le
\exp\Big[\ff{\aa\|\GG_t(x-y)\|^2}{2(\aa-1)}\Big] P_tf^\aa(y),\ \
x,y\in \H, t>0.\end{equation}Consequently, $(\ref{*1})$ implies
$(\ref{*3})$ for any $f\in \BB_b^+(\H).$
 Moreover, for any $f\in
\BB_b(\H)$ and $x,y\in \H$,
\beq\label{Gr}
\beg{split}&|P_t f(x)-P_tf(y)|^2\\
\le & \big(\e^{\|\GG_t(x-y)\|^2}-1\big) \min\big\{P_tf^2(x)
-(P_tf(x))^2, P_tf^2(y) -(P_tf(y))^2\big\}.
\end{split}
\end{equation}
\end{thm}

When $T_t \H\subset R_t^{1/2}\H$,  $\Gamma_{t}$ is a bounded
operator by the closed graph theorem. In this case Theorem
\ref{T1} implies the following result.

\beg{thm}\label{T2} Let $t>0$. The following statements are
gradually weaker, i.e. statement $(i)$ implies statement $(i+1)$ for
$1\le i\le 4$: \beg{enumerate} \item[$(1)$]\ $T_t\H\subset
R_t^{1/2}\H;$
\item[$(2)$]\ $\|\GG_t\|<\infty$ and for any $\aa>1$ and $f\in
\BB_b^+(\H),$ \beq\label{H2}
(P_tf(x))^\aa\le\exp\Big[\ff{\aa(\|\GG_t\|\cdot\|x-y\|)^2}{2(\aa-1)}\Big]
P_tf^\aa(y),\ \ x,y\in \H;\end{equation}\item[$(3)$]\
$\|\GG_t\|<\infty$ and there  exists $\aa>1$ such that $(\ref{H2})$
holds for all $f\in \BB_b^+(\H);$\item[$(4)$]\  $\|\GG_t\|<\infty$
and for any    $f\in \BB_b^+(\H)$ with $f\ge 1,$ \beq\label{LH}
P_t\log f(x)\le \log P_t f(y) +\ff{\|\GG_t\|^2}2 \|x-y\|^2,\ \
x,y\in \H;\end{equation}
\item[$(5)$]\ $P_t$ is strong Feller. \end{enumerate}
 If, in
particular, $M=0,$ then all the above statements are equivalent. 
\end{thm}

According to \cite[Theorem 3.1]{FR00}, $(P_t)$ has a unique invariant
probability measure $\mu$ provided
\ \newline {\bf (A)} \ \emph{$\lim\limits_{t\to \infty} \|T_t x\|=0$
for $x\in \H$; $\sup_{t>0}\operatorname{Tr} R_t<\infty$; $\int_0^\infty \d
s\int_\H (1\land \|T_s x\|^2)M(\d x)<\infty$;
$$ \lim_{t\to\infty} \bigg\{\int_0^t T_s a\,\d s +\int_0^t\d s \int_\H
T_s x \Big(\ff 1 {1+\|T_s x\|^2}-\ff 1 {1+\|x\|^2}\Big)M(\d
x)\bigg\}$$ exists in $\H$; and  $R$ is of trace class.}

In this case, if $P_t$ is strong Feller then it has a  density
$p_t(x,y)$ w.r.t. $\mu$ on $\text{supp}\,\mu$, the support of $\mu$.
As observed in the recent paper \cite{W09}, the Harnack inequality
$(\ref{H2})$ and the log-Harnack inequality $(\ref{LH})$ are
equivalent to the following inequalities for $p_t(x,y)$
respectively:
\beq\label{HH} \int_\H p_t(x,z)\Big( \frac{p_t(x,z)}{p_t(y,z)}
\Big)^{\frac{1}{\alpha-1}} \,\mu(\d z)
                \leq \exp\Big[\frac{\alpha (\|\Gamma_T\|\cdot\|x-y\|)^2}{2(\alpha-1)^2}\Big],\
                \ x,y\in\text{supp}\,\mu;\end{equation}
\beq\label{HLH} \int_\H p_t(x,y)\log
\frac{p_t(x,z)}{p_t(y,z)}\,\mu(\d z)\leq \frac{\|\GG_t\|^2}{2}
\|x-y\|^2,\quad x,y\in\text{supp}\,\mu.\end{equation} 
Moreover, if $M=0$, by
e.g. \cite[Theorem 10.3.5]{DZ02}, Theorem \ref{T2} (1) implies that $P_t
L^p(\H;\mu)\subset \scr C^\infty(\H)$ for $p>1$ and $t>0$. So, we have the
following consequence of Theorem \ref{T2}.

\beg{cor}\label{C3} Let $M=0$ and assume that $(P_t)$ has a 
invariant probability measure $\mu$ with full support.   Then for
any $t>0,$ $(1)$--$(5)$ in Theorem $\ref{T2}$ and the following
statements are equivalent:
\beg{enumerate}
\item[$(6)$] For any $\aa>1$, $(\ref{HH})$ holds; \item[$(7)$] For
some $\aa>1$, $(\ref{HH})$ holds;
\item[$(8)$] The entropy inequality $(\ref{HLH})$ holds;
\item[$(9)$] For any $p>1,$ $P_t
L^p(\H;\mu)\subset \scr C^\infty(\H).$\end{enumerate}\end{cor}

The following result is a standard consequence of the Harnack
inequality (\ref{H2}), where $(i)$  follows from \cite[Proposition
4.1]{DRW09}, $(ii)$ follows from Lemma \cite[Lemma 2.2]{RW03}, and
the proof of $(iii)$ is similar to the those of \cite[Theorem 1.5
and Proposition 1.6]{RW03} (see also \cite{OY} for details).

\begin{cor} Assume that $(P_t)$ has a invariant measure and that
$\GG_t$ is bounded for a fixed $t>0$, 
and let $p_t(x,y)$ be the density of $P_t$ w.r.t. $\mu$.
Then: \beg{enumerate} \item[$(i)$]\ $P_t L^p(\H;\mu)\subset \scr
\C(\H)$ for any $p>1$. \item[$(ii)$]  For any $\aa>1,$
    \[
    \|p_{t}(x,\cdot)\|_{L^{\alpha/(\alpha-1)}(\H;\mu)}\leq
    \left[
        \int_{\H}\exp\left(-\frac{\alpha\|\GG_t\|^2}{2(\alpha-1)}\|x-y\|^{2}\right)\,\mu(dy)
    \right]^{-1/\alpha},\ \ x\in {\rm supp}\mu.
    \]

\item[$(iii)$]  If there exist some constants $\varepsilon >0$ and
$\alpha>1$ such that $$  C(t,\alpha,\varepsilon):=
        \int_{\H} \left[ \int_{\H}
            \exp\left(-\frac{\alpha \|\GG_t\|^2}{2(\alpha-1)} \|x-y\|^{2} \right)
        \,\mu(dy) \right]^{-(1+\varepsilon)} \mu(dx)<\infty,
$$
then $P_t$ is hyper-bounded with
\[  \|P_{t}\|_{\alpha\rightarrow \alpha (1+\varepsilon)}
\leq  C(t,\alpha,\varepsilon)^{\frac{1}{\alpha(1+\varepsilon)}}.
\]
If $C(t,\alpha,0)<\infty$ then $P_t$ is uniformly integrable in
$L^\aa(\H;\mu)$ and hence $P_{s}$ is compact on $L^{\alpha}(\H,\mu)$
for every $s>t$. 
\end{enumerate}
\end{cor}

We shall prove Theorems \ref{T1} and \ref{T2} in the next section,
and  present in Section 3  applications of the log-Harnack
inequality to cost-entropy inequalities of the semigroup and the HWI
inequality in the Gaussian case. Finally, in Section 4 we
investigate the Harnack inequality and strong Feller property for a
class of semi-linear stochastic equations by using a perturbation
argument.

\section{Proofs of Theorems \ref{T1} and \ref{T2}}

As explained in the last section, Corollary \ref{C3} is a direct
consequence of Theorem \ref{T2}. Since (2) implying (3) is trivial,
  (3) implying (4) and (4) implying (5) have been proved in \cite{W09} for
  Markov semigroups on abstract Polish spaces,
  and (5) implying (1) follows from \cite[Theorem 9.19]{DZ92},
  it suffices to prove
  Theorem \ref{T1}.

  Consider the following linear control system on $\H$
 \begin{equation}\label{N}
\d x_{t}=Ax_{t}\,dt+R^{1/2}u_{t}\,dt, \ \
  x_{0}=x\in\H.
 \end{equation}

According 
\cite[Part IV, Theorem 2.3]{Zab08} (ref. also the appendix of \cite{DZ92} or \cite{DZ02}),
\begin{equation}\label{GG}
    \|\Gamma_{t}x\|^{2}=\inf\bigg\{\int_{0}^{t} \|u_{s}\|^{2}\d s:\
            u \in L^{2}([0,t]\to\H; \d s),  x_0=x, x_{t}=0 \bigg\}.
\end{equation}
This implies the following upper bounds of $\|\GG_t x\|.$

\begin{prp}\label{P2.1} Let $t>0$.  Then for any strictly positive
$\xi\in C([0,t])$,
\begin{equation}\label{2.1}
    \| \Gamma_{t} x\|^2 \leq
        \frac{ \int_{0}^{t}  \|R^{-1/2}T_{s} x\|^{2}\,\xi_{s}^{2}    \d s }
        { \left( \int_{0}^{t}   \xi_{s} \d s \right) ^2},\quad x\in\H,
\end{equation}
 where $\|R^{-1/2}x\|=\infty$ if $x\notin R^{1/2}\H.$
Consequently, $(\ref{*1})$ implies
\begin{equation}\label{2.2}
    \| \Gamma_{t}x\|^2 \leq
            \frac{\|R^{-1/2}x\|^2}{ \int_{0}^{t}  h(s)^{-1}\,ds  },\ \ x\in  \H.
\end{equation}
\end{prp}

\begin{proof}
We only need to consider the case that $T_s x\in R^{1/2}\H$ for a.e.
$s\in [0,t]$ and $\{\xi_s R^{-1/2} T_s x\}_{s\in [0,t]}\in
L^2([0,t]\to\H;\d s).$ In this case, for
\begin{equation*}
u_s:=-\frac{\xi_{s}}{\int_{0}^{t} \xi_{r}\,\d r } R^{-1/2}T_{s} x
,\quad s\in[0,t],
\end{equation*}
one has a null control of the system \eqref{N}; that is, $u\in
L^2([0,t]\to\H;\d s)$ and

$$x_t:= T_t x +\int_0^t T_{t-s}R^{1/2} u_s\d s=0.$$ Then (\ref{2.1})
follows from (\ref{GG}) by taking $\xi(s)=h(s)^{-1}$, $s\in[0,t]$.
\end{proof}

To prove the desired Harnack inequality, we adopt the following
Girsanov theorem for L\'evy processes. Let  $\|\cdot\|_0$ be the
norm on $\H_0:=R^{1/2}(\H)$ with inner product
$\<x,y\>_{0}:=\<R^{-1/2}x, R^{-1/2}y\>$ for   $x,y\in \H_0$.

\begin{prp}
\label{P2.2} Let $t>0$. Suppose that $(Z_s)_{0\leq s\leq t}$ is an
$\H$-valued L\'evy process on a filtered probability space $(\Omega,
\F, (\F_{s})_{0\leq s\leq t},\P)$ with characteristic triplet $(a,R, M)$. Denote by
$Z'$  the Gaussian part of  $Z$. For any $\H_0$-valued   predictable
process $\psi_s$, independent of $Z_s-Z'_s$ such that
$$ s\mapsto
\rho_s:=\exp\bigg(\int_0^s \<\psi_r,\d Z'_r\>_{0}-
            \frac12\int_0^s \|\psi_r\|_{0}^2\,\d r  \bigg)$$
            is a $\F_s$-martingale, the process
$$[0,t]\ni s\mapsto  \tt{Z}_s:=Z_s-\int_0^s\psi_r\,\d r $$ is
also a L\'evy process  with   characteristic triplet $(a,R, M)$
under the  probability measure  $\d\tt \P:=\rho_t\d \P.$
\end{prp}
\begin{proof}
We write
\[\E\exp(\i\<Z_s, z\>)=\exp\left[-s\vartheta_{1}(z)-z\vartheta_{2}(z)\right], \quad z\in\H, \]
where
\[\vartheta_{1}(z):=\frac{1}{2}\<Rz,z\>\]
and
\[\vartheta_{2}(z):=-\i\<z,a\>+\int_{\H}
               \left[1-\exp(\i\<z,x\>)+\i\<z,x\>1_{\{|x|\leq 1\}}(x)\right]\,M(\d x).
\]
Correspondingly, the process $Z_s$ is
 decomposed by $Z_s=Z'_s+Z''_s,$
where $Z'_s$ is the Gaussian part of $Z_s$ with symbol
$\vartheta_{1}$, and  $Z''_s$ is the jump process with
symbol $\vartheta_{2}$.

 By the Girsanov theorem for Wiener processes on Hilbert space (see
\cite[Theorem 10.14]{DZ92}),
\[\widetilde{Z}'_s=Z'_s-\int_0^s\psi_r\,\d r, \quad 0\leq s\leq t\]
is   an $R$-Wiener process under the probability measure $\tt{\P}$.
Consequently, for all $0\leq s\leq t$ and all $z\in\H$, by the
martingale property of $\rr_s$ we have
$$
  \E_{\tt\P} \big[ \exp(\i \<z,   \tt{Z}'_s\> \big]
=\E \big[\rr_s \exp(\i \<z,   \tt{Z}'_s\>)\big]=\E\exp(\i \<z,
Z'_s\>)
 =\exp\left[-s\vartheta_{1}(z)\right],$$ where $\E_{\tt\P}$ is the
 expectation taken for $\tt\P.$
Combining this with the independence of $Z'$ and $Z''$, we obtain

   $$ \E_{\tt{\P}}\exp\left(\i\<z,\widetilde{Z}_s\>\right)
=\Big(\E\rho_s\exp\left[\i\left\<z,
\widetilde{Z}'_s\right\>\right]\Big)
     \E\exp\left(\i\<z, Z''_s\>\right)
=\exp\left[-s\vartheta_{1}(z)-t\vartheta_{2}(z)\right].$$ Thus,
under $\tt\P$ the characteristic symbol of $\widetilde{Z}_s$   is
also $\vartheta_{1}+\vartheta_{2}$. This completes the proof.
\end{proof}

By Proposition \ref{P2.2}, we are able to establish the  Harnack
inequality by using the null controllability of the deterministic
equation (\ref{N}).

\beg{prp}\label{P2.3}   Let $t>0$ and $x,y\in\H$. Suppose that there
exists  $u \in L^{2}([0,t]\to \H; \d s)$ such that $x_t=0$, where
$x_s$ solves $(\ref{N})$ with $x_0=y-x$.  Then for any $\aa>1,$
\begin{equation}\label{H'}
    (P_{t}f)^{\alpha}(x)\leq \exp\left(\frac{\alpha}{2(\alpha-1)} \int_{0}^{t} \|u_s\|^{2}\,\d
    s
        \right)  P_{t}f^{\alpha}(y),\ \ f\in \BB_b^+(\H).
\end{equation} Moreover, for any $f\in \BB_b(\H)$ and $x,y\in\H,$
\begin{equation}\label{Gr'}   |P_t f(x)-P_tf(y)|^2\\
\le \big(\e^{\int_{0}^{t} \|u_s\|^{2}\d s}-1\big) \big\{P_tf^2(y)
-(P_tf(y))^2 \big\}.
\end{equation}

\end{prp}

\begin{proof}
Let $(Z'_{s})_{0\leq s\leq t}$ be the Gaussian part of the L\'evy
process $Z_{s}$, which is an $R$-Wiener process on $\H$. Let
$\psi_{s}=R^{1/2}u_{s} \in \H_{0}$ for $s\in [0,t]$.   Then by
Proposition \ref{P2.2},

$$
\widetilde{Z}_{s}:=Z_{s}-\int_{0}^{s} \psi_r\,\d r,\quad 0\leq s\leq
t$$ is  a L\'evy process with characteristic triplet $(a,R,M)$ under
the probability measure   $\tt{\P}$.

Let
\begin{equation*} \beg{split}
&Y_{t}^{y}=S_{t}y+\int_{0}^{t}S_{t-s}\,\d Z_{s},\\
&X_{t}^{x}=S_{t}x+\int_{0}^{t}
S_{t-s}\,\d\widetilde{Z}_{s}.\end{split}
\end{equation*} 
Then, by the
definition of $P_t$ and since $Z_s$ and $\widetilde{Z_s}$ are cylindrical L\'evy
processes with characteristic triplet $(a,R,M)$ under $\P$ and
$\tt\P$ respectively, we have
\begin{equation}\label{PP}P_t f(y)= \E f(Y_t^y),\ \ P_tf(x)= \E_{\tt\P}
f(X_t^x)=\E\big[\rr_t f(X_t^x)\big],\ \ f\in
\BB_b(\H).
\end{equation} 
Moreover, it is easy to see that
 $$X^{x}_{s}=Y_{s}^{y}-x_{s},\ \ \ s\in [0,t].$$
So, $X_t^x=Y_t^y$ due to  $x_{t}=0$. Combining this with (\ref{PP}),
for any $f\in \BB_b^+(\H)$ we have
\begin{equation*}
\begin{split}
\P_{t}f(x)  &=\E  \big[\rho_{t} f(X_{t}^{x})\big]=\E  \big[\rho_{t}
f(Y_{t}^{y})\big]\\
&\leq \bigl(\E
\rho_{t}^{\alpha/(\alpha-1)}\bigr)^{(\alpha-1)/\alpha} \bigl(\E
f^{\alpha}(Y_{t}^{y})\bigr)^{1/\alpha}      = \bigl(\E
\rho_{t}^{\alpha/(\alpha-1)}\bigr)^{\alpha-1)/\alpha}
\bigl(P_{t}f^\alpha(y)\bigr)^{1/\alpha}.
\end{split}
\end{equation*}

This implies (\ref{H'}) by noting that
$$
\E\rho_{t}^{\alpha/(\alpha-1)}= \exp\Big[
                    \frac{\alpha}{2(\alpha-1)^2}\int_{0}^{t} \|\psi_s\|_{0}^{2}\,\d
                    s\Big].$$
Similarly, since $\E\rr_t=1$, we have \begin{equation*}\beg{split}
    |P_{t}f(x)-P_{t}f(y)|^2
&=|\E f(\rho_{t}X_{t}^{x})-\E f(Y_{t}^{y})|^2\\
&=\big|\E \big\{(\rho_{t}-1)
\big(f(Y_{t}^{y})-P_tf(y)\big)\big\}\big|^2\leq \big\{P_t
f^2(y)-(P_tf(y))^2\big\}  \E |\rho_{t}-1|^2.
\end{split}
\end{equation*}
This implies (\ref{Gr'}) by noting that
 $$\E(\rho_{t}-1)^{2} =\E \rho_{t}^{2}-1=\e^{\int_0^t\|u_s\|^2\d
 s}-1.$$  \end{proof}

\ \newline \emph{Proof of Theorem \ref{T1}.} Combining (\ref{GG})
with (\ref{H'}), we obtain (\ref{H1}). If (\ref{*1}) holds, then
(\ref{*3}) follows from (\ref{H'}) according to Proposition
\ref{P2.1}. Finally, (\ref{Gr}) follows from (\ref{GG}) and
(\ref{Gr'}), where the latter holds   also by exchanging the
positions of $x$ and $y$.\qed

\section{Application to the HWI inequality}

The HWI inequality, established in \cite{OV} and reproved in
\cite{BGL} for symmetric diffusions on finite dimensional Riemannian
manifolds, links the entropy, information and the
transportation-cost. In this section, we shall prove it for the
present non-symmetric infinite-dimensional model.

Throughout this section we assume that 
\\
{\bf (A$'$)}\ \
$P_t$ has an invariant
probability measure $\mu$.

This assumption follows from assumption {\bf
(A)} as explained in Section 1. We first observe that the
log-Harnack inequality (\ref{LH}) implies an entropy-cost inequality
for $P_t^*$, the joint operator of $P_t$ on  $L^2(\H;\mu).$

\beg{prp} \label{P3.1} 
Assume $(A')$. 
Let $P_t^*$ be the adjoint operator of $P_t$
on $L^2(\H;\mu)$. If $\|\GG_t\|<\infty$, then
$$\mu((P_t^*f)\log P_t^*f)\le \ff {\|\GG_t\|^2} 2 W_{2}(f\mu,\mu)^2,\ \ \
f\ge 0, \mu(f)=1$$ holds, where $W_{2}^2$ is the Warsserstein
distance induced by the cost-function $(x,y)\mapsto \|x-y\|^2$; that
is,
$$W_{2}(f\mu,\mu)^2 =\inf_{\pi\in \scr
C(f\mu,\mu)}\int_{\H\times\H} \|x-y\|^2\pi(\d x,\d y)$$ for $\scr
C(f\mu,\mu)$ the set of all couplings of $f\mu$ and $\mu$.
Consequently,  $(\ref{*1})$ implies

$$\mu((P_t^*f)\log P_t^*f)\le
\ff{W_2(f\mu,\mu)^2}{2\int_0^th(s)^{-1}\d s},\ \ t>0.$$  \end{prp}

\beg{proof} Due to Proposition \ref{P2.1}, it suffices to prove the
first assertion. We shall adopt an argument in \cite{BGL} by using
the log-Harnack inequality (\ref{LH}). Let $f\ge 0$ such that
$\mu(f)=1$. By an approximation argument, we may assume that $f$ is
bounded. So, by Theorem \ref{T2} we have
$$P_t (\log P_t^*f)(x)\le \log (P_tP_t^*f)(y) +\ff {\|\GG_t\|^2} 2
\|x-y\|^2,\ \ x,y\in \H.$$ 
Integrating both sides w.r.t.
$\pi\in\scr C(f\mu, \mu)$, and minimizing in $\pi$, we arrive at
$$\mu((P_t^* f)\log P_t^* f) \le \mu(\log (P_tP_t^*f)) + \ff
{\|\GG_t\|^2} 2 W_2(f\mu,\mu)^2.$$ This completes the proof by
noting that, since $\mu$ is invariant for $P_t$ and $P_t^*$,
$$ \mu(\log(P_tP_t^*f))\le \log \mu(P_t P_t^*f)=\log
1=0.$$
\end{proof}

According to the above result, to derive the entropy-cost inequality 
for $P_t$,we shall need the log-Harnack inequality for the adjoint semigroup
$P_t^*$. To ensure that $P_t^*$ is again an O-U type semigroup, we
shall simply consider the Gaussian case (i.e. $M=0$), and assume {\bf (A)}. In
this case, $\mu$ is a Gaussian measure with co-variance
$$R_\infty:= \int_0^\infty T_s RT_s^*\d s.$$
To see that $P_t^*$ as a generalized Mehler semigroup (in the sense of \cite{FR00}), we assume
further

\ \newline {\bf (B)} \ \emph{$M=0, R_\infty \H\subset \D(A^*)$, and
the operator $\tt A = R_\infty A^* R_\infty^{-1}$ with domain
$$\D(\tt A):= \big\{x\in R_\infty \H:\ R_\infty^{-1} x \in
\D(A^*)\big\}$$ generates a   $C_0$-semigroup $\tt T_t$ on $\H$ such
that
$$\tt R_t=\int_0^t \tt T_s R\tt T_s^*\d s $$ is of trace class for $t>0.$}

\ \newline In this case, $P_t^*$ can be calculated explicitly as  (see
\cite[Proposition 10.1.9]{DZ02})

$$P_t^* f(x)= \int_\H f(\tt T_t x + y) N_{\tt R_{t}}(dy), \quad f\in
\mathscr{B}_b(\H),$$ where $N_{\tt R_t}$ is the centered Gaussian
measure with co-variance $\tt R_t.$ Thus, $P_t^*$ is the transition semigroup
of the solution to

$$\d \tt X_t= \tt A \tt X_t\d t + R^{1/2}\d W_t$$ for $W_t$ the
cylindrical Brownian motion on $\H$.   So, the following is a direct
consequence of Theorems \ref{T1} and \ref{T2} and Proposition
\ref{P3.1}.

\beg{prp}\label{P3.2} Assume {\bf (A), (B)}. Let $\tt \GG_t= \tt
R^{-1/2}_t \tt T_t.$ Then
$$(P_t^* f )^\alpha (x)\leq P_t^*f^\alpha (y) \exp\left(
\frac{\alpha \|\tt \GG_t(x-y) \|^2}{2(\alpha-1)}    \right),\ \
f\in\BB_b^+(\H), x,y\in\H.$$
If\/ $\tt\GG_t$ is bounded, then
$(\ref{LH})$ holds for $P_t^*, \tt\GG_t$ in place of $P_t$ and
$\GG_t$ respectively. In particular,
\begin{equation}\label{EC} \mu((P_t f)\log P_t f)\le \ff {\|\tt\GG_t\|^2} 2
W_2(\mu, f\mu)^2,\ \ f\ge 0, \mu(f)=1.\end{equation} \end{prp}

Let $W_0$ be the space of functions $f$ of the form
\[  f(x)=F(\<\xi_1,x\>,\cdots,\<\xi_m,x\>),\quad x\in \H
\]
for some $m\geq 1$ and $F\in \scr S (\R^m ,\mathbb{C})$ (i.e. the
Schwartz space of complex-valued functions, $``$rapidly decreasing"
at infinity as well as their derivatives). Let $W$ be the
real-valued elements of $W_0$. According to \cite{LR04}, $W$ is
dense in $L^p(\mu)$ for any $p\geq 1$ and is a core of $D(L)$, the
$L^2(\mu)$-domain of the generator $L$ of $P_t$. We are now able to
present the following result on the HWI inequality.

\beg{thm} \label{T3.3} Assume {\bf (A)} and {\bf (B)}. Assume
further that $A^*$ has a sequence of eigenvectors separating the
points in $\H$. If $(\ref{*1})$ holds then
$$\mu(f^2\log f^2)\le 2\mu(\<\R Df, Df\>)\int_0^t h(s)\d s+ \ff {\|\tt
\GG_t\|^2} 2 W_2(\mu, f^2\mu)^2,\ \ t>0, f\in W, \mu(f^2)=1,$$ where
$Df$ is the Fr\'echet derivative of $f$.
 \end{thm}
\beg{proof} Let $f\in W$ such that $\mu(f^2)=1.$ By \cite[Theorem
1.3(2)]{RW03}, we have
$$P_t(f^2\log f^2)\le (P_t f^2)\log P_t f^2 + 2\bigg(\int_0^t h(s)\d
s\bigg) P_t \<RDf, Df\>.$$ 
Integrating w.r.t. $\mu$ we obtain
$$\mu(f^2\log f^2) \le 2\mu(\<\R Df, Df\>)\int_0^t h(s)\d s+
\mu((P_t f^2)\log P_t f^2).$$ The proof is then completed by
combining this with Proposition \ref{P3.2}. \end{proof}

If in particular $P_t$ is symmetric (it is the case iff
$AR^{1/2}=R^{1/2}A^*$), then (\ref{*1}) implies

$$\mu(f^2\log f^2)\le 2\mu(\<\R Df, Df\>)\int_0^t h(s)\d s+ \ff 1  {2\int_0^t h(s)^{-1}\d s}
W_2(\mu, f^2\mu)^2$$ for all $ f\in W, \mu(f^2)=1, t>0.$


\section{Semi-linear stochastic equations}
  Consider the equation
\begin{equation}\label{4.1} \d X_t^x = AX_t^x\d t + F(X_t^x)\d t + R^{1/2} \d
W_t,\ \ X_0^x=x\in\H,\end{equation} where $F$ is a measurable map on
$\H$ such that $F(\H)\subset R^{1/2} \H,$ and $W_t$ is the
cylindrical Brownian motion on $\H$. We shall establish the Harnack
inequality for the associated semigroup by regarding (\ref{4.1}) as
a perturbation to (\ref{1.1}) with $Z_t= R^{1/2} W_t$, i.e. $b=0,
M=0.$ Since we do not assume that $F$ is dissipative, the study is
not included in \cite{DRW09}. In general, this equation only admits
a weak solution. In this paper we shall consider the weak solution
for \eqref{4.1} constructed from (\ref{1.1}) with $Z_t= R^{1/2} W_t$ by Girsanov
transformations.

Let $\tt X_t^x$ be the mild solution to
$$\d \tt X_t^x = A \tt X_t^x\d t + R^{1/2} \d W_t,\ \ X_0^x=x.$$ We
have $\tt X_t^x= W_A(t) +T_t x,$ where
$$W_A(t) = \int_0^{t} T_{t-s} R^{1/2}\d W_s.$$ Since $\int_0^t T_s RT_s^*\d s$ is of trace class,
$W_A\in L^2([0,t];\H)$ for any $t>0.$ Let
\beg{equation*}\beg{split} &\psi_x(t)=R^{-1/2}F(W_A(t)+T_tx), \\
    &\tt W_t^x=W_t-\int_0^t \psi_x(s)\d s, \\
    &\rho_t^x=\exp\left(\int_0^t\<\psi_x(s),\d W_s\>-\frac{1}{2} \int_0^t \|\psi_x(s)\|^2\d
    s\right).\end{split}
\end{equation*}

Assume that
\begin{equation}\label{4.2} \|R^{-1/2} F(x)\|^2\le k_1 +k_2 \|x\|^2,\ \
x\in\H\end{equation} holds for some $k_1,k_2\ge 0.$ Then  by
\cite[Theorem 10.20]{DZ92} and its proof, $\Q_x:= \rr_t^x\P$ is a
probability measure and $\tt X_t^x$ is a weak  solution to
(\ref{4.1}) under $\Q_x$    with respect to the cylindrical Brownian
motion  $\tt W_t^x$. Denote the corresponding $``$semigroup" by
\begin{equation}\label{PPP}P_t^Ff(x) = \E_{\Q_x} f(\tt X_t^x),\ \ \ f\in
\BB_b(\H).\end{equation} We note that due to the lack of uniqueness,
in general $P_t^F$ may not provide a semigroup (but cf. also \cite{LR04}).
 Let $P_t$ be the semigroup of $\tt X_t^\cdot$ under $\P$.
By Theorem \ref{T1} we have
\begin{equation}\label{4.3}
    (P_tf)^\alpha(x)\leq P_tf^\alpha (y)
                        \exp\left( \frac{\alpha\| \Gamma_t(x-y) \|^2}{2(\alpha-1)} \right)
    , \quad f\in \BB_b^+(\H),
\end{equation}
where $\Gamma_t:=R_t^{-1/2}T_t$. Moreover, by \cite[(10.42)]{DZ92},
for any $p>0$ there exists $t_p>0$ such that
$$C_{p,k_2}(t):= \E \exp\bigg(2p(2p+1)k_2\int_0^t \|W_A(s)\|^2\d
s\bigg)<\infty,\ \ t\in [0,t_p].$$ In particular, if $k_2=0$ then
$C_{p,k_2}(t)=1,\ t\ge 0.$ More precisely, let
$$\theta=\text Tr \int_0^1 T_sRT_s^*\d s.$$ We have
$$C_0:=\sup_{s\in [0,1]}\E \e^{\|W_A(s)\|^2/4\theta}<\infty.$$ Thus,
for any $\ll>0$, \begin{equation}\label{GGG} \beg{split} \E \e^{\ll \int_0^t
\|W_A(s)\|^2\d s} &=\E\e^{\ff 1 t \int_0^t \ll t \|W_A(s)\|\d s} \le
\ff 1 t \int_0^t \E \e^{\ll
t\|W_A(s)\|^2}\d s\\
&\le \ff 1 t \int_0^t \big(\E\e^{\|W_A(s)\|^2/4\theta}\big)^{4\theta
\ll t}\d s \le C_0^{4\theta\ll t},\ \ \
t\in [0,1\wedge (4\theta\ll)^{-1}].\end{split}
\end{equation}

Combining this with (\ref{PPP}) we obtain the following result.
\begin{thm} If $(\ref{4.2})$ holds, then for any $t>0$,  $\alpha>1$, $x,y\in \H$, $p, q>1$ with $\alpha/(pq)>1$,
and  $f\in \BB_b^+(\H)$
\begin{equation*}
\begin{aligned}
    (P_t^F f)^\alpha(x)\leq & \left(C_{\frac{p}{p-1},k_2}(t)\right)^{\aa p/(2(p-1))}
                         \left(C_{\frac{1}{q-1}, k_2}(t)\right)^{\aa q/(2(q-1))}    P_t^F f^\alpha (y)
                        \exp\left( \frac{\alpha q \| \Gamma_t(x-y)\|^2}{2(\alpha-q)} \right.
                          \\
                         &\left.\quad + \aa \left[\frac{p+1}{p-1}+\frac{q+1}{q(q-1)}\right]\int_0^t
                         \left[k_1+k_2(\|T_s x\|^2+\|T_s
                         y\|^2)\right]\d s\right).
\end{aligned}
\end{equation*}  Consequently, if $\|\GG_t\|<\infty$ for $t>0$, then
$P_t^F$  is strong Feller provided it is a semigroup.
\end{thm}

 \beg{proof}
    For simplicity, we denote
        $p'=\frac{p}{p-1}$,         $q'=\frac{q}{q-1}$,     $\theta=\alpha/(pq)$.
\[
\begin{aligned}
    P_t^F f (x) &= \E_{\Q_x} f(\tt X_t^x)=\E \rho_t^x f(\tt X_t^x)
                \leq  (\E f^p (\tt X_t^x))^{1/p} (\E (\rho_t^x)^{p'} )^{1/p'}\\
            & =  (P_t  f^p (x))^{1/p} (\E(\rho_t^x)^{p'} )^{1/p'}\\
             &   \leq \left[    P_t f^{\theta p}(y)
                    \exp\left(  \frac{\theta\|\Gamma_t(x-y)\|^2}{2(\theta-1)}   \right)
                        \right]^{1/(\theta p)} (\E(\rho_t^x)^{p'} )^{1/p'}.
\end{aligned}
\]
On the other hand, for any $g\in \C_b^+(\H)$,
\[
    P_tg(y)\leq \E_\P g(\tt X_t^y) =\E_{\Q_y}g(\tt X_t^y) (\rho_t^y)^{-1}
    \leq (P_t^F g^q (y))^{1/q} (\E(\rho_t^y )^{1-q'} )^{1/q'}.
\]
So by taking $g=f^{\theta p}$
\[
    ( P_t^F f )^\alpha (x)
        \leq P_t^F f^{\alpha}(y) \exp\left( \frac{\alpha \|\Gamma_t(x-y)\|^2}{2p(\theta-1)}   \right)
        (\E(\rho_t^x)^{p'} )^{\alpha/p'}
         (\E (\rho_t^y)^{1-q'} )^{\alpha/q'}.
\]
This implies the desired Harnack inequality according to the
following Lemma \ref{L4.2}.

Now, assume that $\GG_t$ is bounded for $t>0$ and assume that
$P_t^F$ is a semigroup. Let $f\in \BB_b^+(\H).$ By the first
assertion and (\ref{GGG}), for any $\aa>1$ there exist constants
$t_\aa, c_\aa>0$ and a positive function $H_\aa$ on $(0,t_\aa)$ such
that
\begin{equation}\label{LLL}P_t^Ff(x)\le (P_t^Ff^\aa(y))^{1/\aa} \e^{c_\aa t+
\|x-y\|^2H_\aa(t)},\ \ \ t\in (0,t_\aa].\end{equation} Then, for any
$t>0$, \beg{equation*}\beg{split}\limsup_{x\to y} P_t^Ff(x)&\le
\limsup_{\aa\to 1}\limsup_{s\to 0}\limsup_{x\to y}\big\{
P_s^F(P_{t-s}^Ff)^\aa(y)\big\}^{1/\aa} \e^{c_\aa s
+\|x-y\|^2H_\aa(s)}\\
&\le \limsup_{\aa\to 1}\limsup_{s\to 0}\limsup_{x\to y}\big\{
P_t^F f^\aa(y)\big\}^{1/\aa} \e^{c_\aa s +\|x-y\|^2H_\aa(s)}= P_t^F f
(y).\end{split}\end{equation*}
On the other hand, (\ref{LLL}) also
implies
\beg{equation*}\beg{split}P_t^Ff(x)&\ge \big\{P_s^F
(P_{t-s}^Ff)^{1/\aa}(y)\big\}^\aa\e^{-\aa c_\aa s-\aa
H_\aa(s)\|x-y\|^2}\\
&\ge \big\{P_t^Ff^{1/\aa}(y)\big\}^\aa\e^{-\aa c_\aa s-\aa
H_\aa(s)\|x-y\|^2}, \ \ \ s\in (0,t_\aa).\end{split}\end{equation*}
So, by first letting $x\to y$ then $s\to 0$ and finally $\aa\to 1$
we arrive at
$$\liminf_{x\to y} P_t^Ff(x)\ge P_t^Ff(y).$$ Therefore, $P_t^Ff$ is
continuous on $\H$.
\end{proof}

\beg{lem} \label{L4.2} Assume $(\ref{4.2})$. For any $p>1$, $\dd>0$
and $x\in\H$, then
\begin{equation*} \beg{split} &\E (\rr_t^x)^p \le
(C_{p,k_2}(t))^{1/2}
\exp\Big( \ff { p(2p-1)} 2\int_0^t(k_1+2k_2\|T_sx\|^2)\d s\Big)\\
&\E(\rr_t^x)^{-\dd}\le  (C_{\dd,k_2}(t))^{1/2} \exp\Big( \ff {
\dd(2\dd+1)} 2 \int_0^t(k_1+2k_2\|x\|^2)\d
s\Big).\end{split}\end{equation*}
\end{lem}

\begin{proof} According to the proof of \cite[Theorem 10.20]{DZ92},
for any $\ll\in \R$, the process

$$t\mapsto \exp\bigg[\ll \int_0^t \<\psi_x(s),dW_s\>- \ff {\ll^2} 2 \int_0^t \|\psi_x(s)\|^2\d
s\bigg]$$ is  a martingale. So,

\[
    \begin{aligned}
           &\E(\rho_t^x)^p\\
        =&\E \exp\left(p  \int_0^t\<\psi_x(s),dW_s\>- p^2 \int_0^t \|\psi_x(s)\|^2\d s\right)
                \exp\left(  \frac{p(2p-1)}{2} \int_0^t \|\psi_x(s)\|^2\d s \right)\\
        \leq & \left[ \E \exp\left(2 p  \int_0^t\<\psi_x(s),dW_s\>
                -2 p^2 \int_0^t \|\psi_x(s)\|^2ds\right) \right]^{1/2}\\
            & \quad\quad\quad\quad\quad\quad\quad\quad\
            \cdot \left[
                \E \exp\left(   p(2 p-1)  \int_0^t \|\psi_x(s)\|^2ds \right)
            \right]^{1/2}\\
            = & \left[\E \exp\left(   p(2 p-1)  \int_0^t \|\psi_x(s)\|^2ds \right)
            \right]^{1/2}.
    \end{aligned}
\]
This implies the first inequality since (\ref{4.2}) and the
boundedness  of $T_s$ imply
$$\|\psi_x(s)\|^2\le k_1 + 2k_2\|W_A(s)\|^2 +2k_2
\|x\|^2.$$
Similarly, the second inequality follows by noting that
\[
    \begin{aligned}
           &\E(\rho_t^x)^{-\dd}\\
        =&\E \exp\left(-\dd  \int_0^t\<\psi_x(s),dW_s\>- \delta^2 \int_0^t \|\psi_x(s)\|^2\d s\right)
                \exp\left(  \frac{\dd(2\dd+1)}{2} \int_0^t \|\psi_x(s)\|^2\d s \right)\\
        \leq & \left[ \E \exp\left(-2 \dd  \int_0^t\<\psi_x(s),dW_s\>
                -2 \dd^2 \int_0^t \|\psi_x(s)\|^2ds\right) \right]^{1/2}\\
            & \quad\quad\quad\quad\quad\quad\quad\quad\
            \cdot \left[
                \E \exp\left(   \dd(2 \dd+1)  \int_0^t \|\psi_x(s)\|^2ds \right)
            \right]^{1/2}\\
            = &  \left[
                \E \exp\left(   \dd(2 \dd+1)  \int_0^t \|\psi_x(s)\|^2ds \right)
            \right]^{1/2}.
    \end{aligned}
\]
\end{proof}
 \beg{thebibliography}{99}


\bibitem{ATW06}
M. Arnaudon, A. Thalmaier, F.-Y. Wang, \emph{Harnack inequality and
heat kernel
  estimates on manifolds with curvature unbounded below,} Bull. Sci. Math.
  130(2006), 223--233.

\bibitem{BGL}
S.G. Bobkov, I. Gentil, M. Ledoux, \emph{Hypercontractivity of
{H}amilton-{J}acobi
  equations,} J. Math. Pures Appl. (9) 80(2001), 669--696.

\bibitem{DRW09}
G. Da Prato, M. R\"ockner, F.-Y. Wang, \emph{Singular stochastic
equations on Hilbert
  spaces: Harnack inequalities for their transition semigroups,} J. Funct.
  Anal. 257(2009), 992--1017.

\bibitem{DZ92}
G.~Da~Prato, J.~Zabczyk, \emph{Stochastic equations in infinite
dimensions,} Vol.~44
  of Encyclopedia of Mathematics and its Applications, Cambridge University
  Press, Cambridge, 1992.

  \bibitem{DZ02}
G.~Da~Prato and J.~Zabczyk, \emph{Second order partial differential
equations
  in {H}ilbert spaces}, London Mathematical Society Lecture Note Series, vol.
  293, Cambridge University Press, Cambridge, 2002.

\bibitem{FR00}
M.~Fuhrman and M.~R{\"o}ckner, \emph{Generalized {M}ehler semigroups: the
  non-{G}aussian case}, Potential Anal. \textbf{12} (2000), no.~1, 1--47.

\bibitem{LR04}
P.~Lescot and M.~R{\"o}ckner, \emph{Perturbations of generalized
{M}ehler
  semigroups and applications to stochastic heat equations with {L}\'evy noise
  and singular drift}, Potential Anal. 20(2004), 317--344.

\bibitem{LW08} W. Liu and F.-Y. Wang, \emph{Harnack inequality and strong
 Feller property for stochastic fast-diffusion equations,} J. Math. Anal. Appl. 342(2008), 651-662.

\bibitem{OV} F. Otto and C. Villani, \emph{Generalization of an inequality
by Talagrand and links with the logarithmic Sobolev inequality,} J.
Funct. Anal.  173(2000), 361--400.

 \bibitem{OY}
S.-X. Ouyang, \emph{Harnack inequalities and applications for
stochastic equations,}
  Ph.D. thesis, Bielefeld University, available on  http://bieson.ub.uni-
  bielefeld.de/volltexte/2009/1463/pdf/ouyang.pdf (2009).

\bibitem{RW03}
M.~R{\"o}ckner, F.-Y. Wang, \emph{Harnack and functional
inequalities for generalized
  {M}ehler semigroups,} J. Funct. Anal. 203(2003), 237--261.

\bibitem{W97}
F.-Y. Wang, \emph{Logarithmic {S}obolev inequalities on noncompact
{R}iemannian
  manifolds,} Probab. Theory Related Fields 109(1997), 417--424.

\bibitem{W07}
F.-Y. Wang, \emph{Harnack inequality and applications for stochastic
generalized
  porous media equations,} Ann. Probab. 35(2007), 1333--1350.

\bibitem{W09} F.-Y. Wang, 
 \emph{Heat kernel inequalities for convexity of 
manifold and curvature condition,} 2009.

\bibitem{Zab08}
J.~Zabczyk, \emph{Mathematical Control Theory,} Modern Birkh\"auser
Classics,
  Birkh\"auser Boston Inc., Boston, MA, 2008, an introduction, Reprint of the
  1995 edition.

 \end{thebibliography}

\end{document}